\documentclass[11pt]{article}
\usepackage{graphicx}
\usepackage{amsmath}
\usepackage{amssymb}
\usepackage{amsthm}
\usepackage{color}
\usepackage{booktabs}
\usepackage{multirow}
\usepackage{amssymb}
\usepackage[linesnumbered,ruled,vlined]{algorithm2e}
\usepackage[title,titletoc]{appendix}
\usepackage{algpseudocode}
\usepackage{rotating}
\usepackage{hyperref,cite}
\usepackage{cleveref}
\usepackage{caption}
\usepackage{accents}
\usepackage{tcolorbox}
\usepackage[bottom]{footmisc}
\usepackage{authblk}
\usepackage{longtable}
\usepackage{graphicx}   
\usepackage{subcaption} 
\numberwithin{equation}{section}
\newcommand{\R}{\mathbb{R}}

\newcommand{\bfn}{{\bf 1}_n}

\newtheorem{theorem}{Theorem}
\newtheorem{definition}{Definition}

\newtheorem{lemma}{Lemma}

\title{Scalable Kernel Quantile Regression: A Preconditioned Augmented Lagrangian Method}

\author{Shengxiang Deng\thanks{School of Data Science, Fudan University, Shanghai, China \texttt{(sxdeng21@m.fudan.edu.cn)}.}, Xudong Li\thanks{School of Data Science, Fudan University, Shanghai, China \texttt{(lixudong@fudan.edu.cn)}.}, Yangjing Zhang\thanks{State Key Laboratory of Mathematical Sciences, Academy of Mathematics and Systems Science, Chinese Academy of Sciences, Beijing, China \texttt{(yangjing.zhang@amss.ac.cn)}.}}

\begin{document}
	\maketitle
    \begin{abstract}
        Kernel quantile regression (KQR) extends classical quantile regression to nonlinear settings using kernel methods, offering a powerful tool for modeling conditional distributions.
		However, its application to large-scale datasets remains challenging due to two intrinsic difficulties: the nonsmoothness of the quantile check loss and the computational burden imposed by the large, dense kernel matrix. Existing state-of-the-art solvers often struggle to handle both challenges simultaneously, leading to limited scalability and high computational cost.
		In this paper, we propose PALM-KQR, a highly efficient two-phase preconditioned  augmented Lagrangian method for large-scale KQR.  In the first phase, an inexact alternating direction method of multipliers (ADMM) is employed to compute a warm-start solution efficiently. The second phase refines this solution using an efficient semismooth Newton augmented Lagrangian method (ALM).
		Our key innovations include a dual
semismooth Newton approach for handling the nonsmooth quantile check loss, and a specialized preconditioning strategy based on low-rank approximations of the kernel matrix, which exploits its structure and mitigates ill-conditioning in the linear systems arising in ALM, thereby significantly accelerating the iterative solvers. Extensive numerical experiments demonstrate that PALM-KQR substantially outperforms existing commercial and specialized KQR solvers in both efficiency and scalability.
    \end{abstract}
    

\section{Introduction} \label{sec:intro}
    Quantile regression (QR) \cite{koenker1978regression} is a fundamental statistical modeling technique that characterizes the conditional distribution of a response variable beyond its conditional mean. While mean regression estimates conditional expectations, QR estimates conditional quantiles at arbitrary levels. This flexibility is crucial for capturing heteroscedasticity, heavy-tailed distributions, and complex conditional dependencies.  As a result, QR has found extensive applications across diverse fields, including risk management \cite{umar2022does}, economics \cite{hendricks1992hierarchical,koenker2001quantile}, 
    healthcare \cite{xu2019does}, and environmental science \cite{long2023esg,pernigo2025probabilistic}, where understanding the full distribution of outcomes is critical for informed decision-making.

    Classical QR is typically formulated with a linear predictor, which limits its ability to capture nonlinear effects.
    Kernel quantile regression (KQR) \cite{takeuchi2006nonparametric,li2007quantile} extends  linear QR by employing kernel methods to implicitly map inputs into a reproducing kernel Hilbert space (RKHS). This nonparametric approach allows KQR to capture complex nonlinear relationships without requiring explicit feature engineering.

    Given data $\{(x_i,y_i),i=1,\dots,n\}$ with inputs $x_i\in\mathbb{R}^p$ and responses $y_i\in \mathbb{R}$, KQR estimates the $\tau$-th quantile of the conditional distribution of $y$ given $x$. A standard formulation of KQR (see, e.g., \cite[(4)]{li2007quantile}) is
    \begin{equation}\label{def:KQR}
    	\min_{\beta\in\mathbb{R},\,f\in\mathcal{H}_k}\sum_{i=1}^n\rho_{\tau}(y_i-\beta-f(x_i))+\frac{\lambda}{2}\|f\|_{\mathcal{H}_k}^2,
    \end{equation}
    where $\mathcal{H}_k$ is the RKHS induced by a positive definite kernel function $k(\cdot,\cdot)$. Typical choices include the radial basis kernel $k(x,x')={\exp}(-\gamma\|x-x'\|_2^2)$ and the Laplacian kernel $k(x,x')={\exp}(-\gamma \|x-x'\|_1)$, with a scale parameter $\gamma>0$.
    The regularization parameter $\lambda > 0$ controls model complexity, and $\tau \in (0,1)$ is the target quantile level. The quantile check loss function $\rho_{\tau}(\cdot)$ and its conjugate function $\rho_{\tau}^*(\cdot)$ are given by
    \begin{equation*}
    	\rho_{\tau}(z) =
    	\begin{cases}
    		\tau z, & \mbox{if } z>0, \\
    		-(1-\tau)z, & \mbox{if } z\leq 0,
    	\end{cases}\quad  \rho_{\tau}^*(v) = \delta_{[\tau - 1,\tau]}(v) =
    	\begin{cases}
    		0, & \mbox{if } \tau - 1 \leq v \leq \tau, \\
            +\infty, & \mbox{otherwise},
    	\end{cases}
    \end{equation*}
	where $\delta_{[\tau - 1,\tau]}(\cdot)$ is the indicator function of the interval $[\tau - 1,\tau]$.
    By the representer theorem \cite{kimeldorf1971some,wahba1990spline}, the optimal solution $f \in \mathcal{H}_k$ admits a finite-dimensional representation $f(x) = \frac{1}{\lambda}\sum_{j=1}^n \theta_j k(x_j, x)$ for some coefficients $\theta = (\theta_1, \dots, \theta_n)^\top \in \mathbb{R}^n$, and the squared RKHS norm simplifies to $\|f\|_{\mathcal{H}_k}^2 = \frac{1}{\lambda^2} \theta^\top K \theta$, where $K$ is the $n \times n$ kernel matrix with entries $K_{ij} = k(x_i, x_j)$. This representation reformulates \eqref{def:KQR} as a finite-dimensional regularized optimization problem (see \eqref{primal-1}). Finally, when a linear kernel is used, KQR in \eqref{def:KQR} estimates a linear regression function, just as linear QR does.

    To obtain this finite-dimensional formulation, we first extend the definition of $\rho_{\tau}(\cdot)$ to vector arguments as $\rho_{\tau}(z) = \sum_{i=1}^{n} \rho_{\tau}(z_i)$, where $z=(z_1,\dots,z_n)^{\top}\in\mathbb{R}^n$. Similarly, its conjugate function extends to vectors as $\rho_{\tau}^*(v) = \delta_{\mathcal{B}}(v)$, where $v\in\mathbb{R}^n$ and $\mathcal{B}:=[\tau - 1,\tau]^n$ represents an $n$-dimensional box.
    Denoting $y:=(y_1,\dots,y_n)^{\top}$ and ${\bf 1}_n$ as the vector of all ones in $\mathbb{R}^n$, we can reformulate problem \eqref{def:KQR} as
    \begin{equation}\label{primal-1}  \operatornamewithlimits{minimize}_{\beta\in\mathbb{R},\,\theta\in\mathbb{R}^n,\,z\in\mathbb{R}^n} \quad \rho_{\tau}(z) + \frac{1}{2\lambda} \theta^{\top} K \theta \quad
    	\mbox{subject to} \quad z = y - \beta {\bf 1}_n - \frac{1}{\lambda} K\theta.
    \end{equation}
	The computational difficulty of KQR is already visible from this formulation. On the one hand, the quantile check loss is nonsmooth, so algorithms that rely directly on smooth gradients or classical Newton steps cannot be applied without modification. On the other hand, kernelization introduces the dense kernel matrix $K$, whose storage, matrix-vector products, and associated linear systems become increasingly demanding as $n$ grows. Moreover, kernel matrices often exhibit spectral decay \cite{williams2000effect,ma2017diving,altschuler2023kernel}, which can lead to severe ill-conditioning in the linear systems arising from high-accuracy second-order methods. 
    The corresponding dual of \eqref{primal-1} is \cite[(6)]{takeuchi2006nonparametric}
    \begin{equation}\label{dual-0}
    	\operatornamewithlimits{maximize}_{\alpha\in\mathbb{R}^n} \quad - \frac{1}{2\lambda} \alpha^{\top} K \alpha + y^{\top}\alpha - \delta_{\mathcal{B}}(\alpha) \quad
    	\mbox{subject to} \quad {\bf 1}_n^{\top} \alpha = 0.
    \end{equation}
    Since $K$ is positive definite, problem \eqref{dual-0} admits a unique solution.
    In this dual form, the nonsmooth check loss appears through the box indicator $\delta_{\mathcal{B}}(\cdot)$, while the kernel matrix enters the dense quadratic term $\alpha^{\top}K\alpha$. We further introduce an auxiliary variable $v$ to decouple the nonsmooth difficulty associated with $\delta_{\mathcal{B}}$ from the computational difficulty caused by the kernel quadratic term involving $K$. This leads to an equivalent consensus form:
    \begin{equation}\label{dual-1}
    	\operatornamewithlimits{minimize}_{\alpha\in\mathbb{R}^n,\,v\in\mathbb{R}^n} \quad \frac{1}{2\lambda} \alpha^{\top} K \alpha - y^{\top}\alpha + \delta_{\mathcal{B}}(v) \quad
    	\mbox{subject to} \quad {\bf 1}_n^{\top} \alpha = 0,\quad \alpha - v = 0.
    \end{equation}
    This formulation provides a structured basis for developing scalable optimization algorithms for solving KQR.

    In the literature, several approaches have been developed for solving KQR. Since the dual formulation \eqref{dual-0} is a convex quadratic program (QP), it can be handled by general interior-point QP solvers. A state-of-the-art implementation of this approach is provided by the R package \texttt{kernlab} \cite{karatzoglou2004kernlab}. Interior-point methods can solve the exact nonsmooth formulation through its QP structure, but their reliance on large dense linear algebra makes them difficult to scale when the kernel matrix is large. Meanwhile, Li et al. \cite{li2007quantile} and Takeuchi et al. \cite{takeuchi2009nonparametric} proposed path-following algorithms for solving KQR. These methods characterize the solution path as the regularization parameter $\lambda$ or the quantile level $\tau$ varies, but they typically rely on certain non-singularity assumptions. Moreover, when only a single parameter value of $\lambda$ or $\tau$ is of interest, their computational efficiency is often inferior to that of interior-point methods, and they are not readily scalable to large problems.

    Another line of work focuses on smoothing and approximation techniques in order to address the nonsmooth quantile check loss $\rho_\tau(\cdot)$. The works of \cite{fernandes2021smoothing, tan2022high, he2023smoothed} proposed several smoothing loss functions to replace the original nonsmooth quantile check loss, thereby enabling the use of standard gradient-based methods. However, such methods aim to solve approximated problems of KQR rather than the exact formulation. Along this line, the \texttt{fastkqr} package \cite{tang2024fastkqr} proposed a framework that solves the original KQR problem through a sequence of smoothing subproblems. Each smoothed subproblem is handled by an accelerated proximal gradient (APG) method, while the smoothing parameter is gradually decreased. Although each APG iteration for a smoothed subproblem costs only $\mathcal{O}(n^2)$, the practical efficiency of the overall framework depends sensitively on the tuning of the smoothing parameter and is further limited by an initial singular value decomposition (SVD), which costs $\mathcal{O}(n^3)$, together with the sublinear convergence of APG. There are also existing works that employ kernel approximation or random features to construct surrogate models for large-scale KQR. For example, \cite{wang2024optimal} used random features to perform stochastic dimensionality reduction. Nevertheless, this line of work mainly emphasizes statistical learning properties rather than the high-accuracy optimization of the original KQR model.

	Overall, current methods still fall short in handling the two fundamental obstacles in KQR simultaneously. Methods that preserve the exact nonsmooth formulation often suffer from the cost of dense kernel linear algebra, whereas smoothing or feature-approximation methods improve tractability by modifying either the loss or the kernel representation. This motivates algorithms that solve the original KQR model while exploiting the special structures of both the box-type nonsmoothness and the kernel matrix.

	In this paper, we propose PALM-KQR, a two-phase preconditioned augmented Lagrangian method (ALM) for solving the dual formulation \eqref{dual-1}. The method addresses the nonsmoothness through a semismooth Newton ALM that exploits the projection structure induced by the box constraint $v \in \mathcal{B}$. It also tackles the kernel-matrix bottleneck through a specialized preconditioner based on a low-rank approximation of $K$, which significantly accelerates the solution of the linear systems arising within the ALM subproblems.
    Briefly, given $\sigma>0$, the augmented Lagrangian function associated with \eqref{dual-1} is
    \begin{equation}
    	\label{lagrangianfunction}
    	\begin{aligned}
    		&L_{\sigma}(\alpha,v;\beta,z)  = \frac{1}{2\lambda} \alpha^{\top} K \alpha - y^{\top}\alpha + \delta_{\mathcal{B}}(v) + {\bf 1}_n^{\top} \alpha \beta + z^{\top}(\alpha-v) + \frac{\sigma}{2}({\bf 1}_n^{\top} \alpha)^2 + \frac{\sigma}{2}\|\alpha - v\|_2^2,
    	\end{aligned}
    \end{equation}
    where $\alpha,v,z \in \mathbb{R}^n,\ \beta\in \mathbb{R}$. The ALM  iteratively addresses a sequence of unconstrained subproblems, progressively approximating the original constrained problem; see~\cite{hestenes1969multiplier,powell1969method,rockafellar1976augmented}. For a nondecreasing sequence of parameters $\sigma_k>0$ and an initial primal variable $(\beta^0,z^0)$, the ALM generates the primal iterative sequence $\{(\beta^k,z^k)\}$ and the dual iterative sequence $\{(\alpha^k,v^k)\}$ as follows:
    \begin{align}
    	(\alpha^{k+1},v^{k+1}) & \approx
    	\operatornamewithlimits{argmin}_{\alpha,v} \, L_{\sigma_k}(\alpha,v;\beta^k,z^k), \label{ALM-subproblem} \\
    	(\beta^{k+1},z^{k+1}) & =  (\beta^k,z^k) + \sigma_k \nabla_{(\beta,z)} L_{\sigma_k}(\alpha^{k+1},v^{k+1};\beta^k,z^k), \label{ALM-multiplier}
    \end{align}
    where $\nabla_{(\beta,z)} L_{\sigma}$ denotes the gradient of $L_{\sigma}$ with respect to $(\beta,z)$.
    However, the quadratic terms and composite structures in \eqref{lagrangianfunction} make updating $(\alpha,v)$ extremely challenging and computationally expensive, particularly for large-scale problems.
    For problems analogous to \eqref{dual-1}, ALM can nevertheless produce high-accuracy solutions efficiently when its subproblems are solved by a semismooth Newton method.
    Recent research \cite{yang2015sdpnal,li2018highly,zhao2010newton,liang2022qppal} has shown that this approach achieves fast convergence when initialized within the fast local convergence region of the method.
    Building on these insights, the first phase of PALM-KQR employs an inexact alternating direction method of multipliers (ADMM) to solve problem \eqref{dual-1} to low or moderate accuracy, providing a sufficiently good initial point for the second phase.
    The second phase of PALM-KQR then refines this solution to high accuracy using a preconditioned semismooth Newton ALM, which leverages the high-quality initial point for fast local convergence.

    The main computational bottleneck in PALM-KQR is the repeated solution of the large-scale symmetric positive definite linear systems arising at each semismooth Newton step. These systems take the form
    \begin{equation*}
    	A_kx=b_k,\quad A_k=K+ \gamma_k {\bf 1}_n{\bf 1}_n^{\top} + L_k ,
    \end{equation*}
    where $K$ is the kernel matrix,  $\gamma_k > 0$ is a scalar parameter, and $L_k$ is a diagonal, positive definite matrix that varies across iterations. For large-scale problems where direct factorization is computationally prohibitive, the Conjugate Gradient (CG) method is a natural choice. However, the convergence rate of CG is governed by the condition number $\kappa(A_k)$. In KQR, $A_k$ is typically ill-conditioned due to the spectral decay of the kernel matrix $K$ \cite{williams2000effect,ma2017diving,altschuler2023kernel}, particularly in high-dimensional settings. This results in a large condition number that severely degrades standard CG performance. To overcome this challenge, we develop a specialized preconditioning strategy. While recent Nystr\"{o}m-based methods \cite{chen2025preconditioning,zhao2022nysadmm,diaz2023robust,chu2024randomized} have proven effective for static systems of the form $K + \mu I$, they are not directly applicable to the dynamic, iteration-dependent systems encountered in our framework. We extend these techniques to handle the specific structure of $A_k$. We construct the preconditioner $P_k$ by applying a Randomly Pivoted Cholesky (RPCholesky) low-rank approximation solely to the ill-conditioned kernel matrix $K$, while retaining the structured terms $\gamma_k {\bf 1}_n{\bf 1}_n^{\top} + L_k$ exactly. Crucially, since $L_k$ is diagonal, the term $\gamma_k {\bf 1}_n{\bf 1}_n^{\top} + L_k$ is a diagonal plus rank-one matrix. This specific structure allows for the efficient computation of the preconditioner's inverse $P_k^{-1}$ via the Woodbury matrix identity (as detailed in Lemma \ref{lemma:preconditioner}), making the strategy computationally viable. This approach theoretically guarantees a bounded condition number for the preconditioned system, ensuring rapid convergence (see Theorem~\ref{thm:pcg rate}). As shown in the numerical experiments, this tailored preconditioner yields significant speedups, particularly for high-dimensional data where standard methods struggle.

    The remainder of this paper is organized as follows. In Section~\ref{sec:two-phase method}, we present the two-phase PALM-KQR framework for solving the dual problem \eqref{dual-1}. Section~\ref{sec:pre} details the construction and theoretical analysis of the effective preconditioner designed to accelerate the solution of the linear systems.
    In Section~\ref{sec:num}, we report numerical experiments to demonstrate the performance of our algorithms.
    Finally, we conclude the paper in Section~\ref{sec:conclusion}.

    \noindent {\bf Notation}:
    We define $[n] := \{1, 2, \dots, n\}$. The set $\mathbb{N}$ denotes the  nonnegative integers. The vector of all ones in $\mathbb{R}^n$ is denoted by $\mathbf{1}_n$, and $I$ denotes the identity matrix of appropriate dimension.
    Given two matrices $H$ and $A$, we write  $H \succeq A$ if and only if $H - A$ is positive semidefinite. 
    Given a positive definite matrix $A$ and a vector $x$, the $A$-norm of $x$ is defined as $\|x\|_A=\sqrt{x^{\top}Ax}$.
    $\lambda_i(A)$ denotes the $i$-th largest eigenvalue of a symmetric matrix $A$. We use ${\rm tr}(A)$ to denote the trace of a matrix $A$. We use $\kappa (A)$ to denote the condition number of a positive definite matrix $A$.
    We use $\lfloor A\rfloor_r$ to denote the best rank-$r$ approximation of a positive semidefinite matrix $A$, obtained via an $r$-truncated eigendecomposition.


	\section{PALM-KQR: a two-phase method for solving  \eqref{dual-1}}\label{sec:two-phase method}
We propose a robust two-phase framework, referred to as PALM-KQR, for efficiently solving the KQR dual problem \eqref{dual-1}.
This approach is motivated by the complementary strengths of first-order and second-order optimization strategies. The dual problem \eqref{dual-1} involves a non-differentiable term arising from the box constraints and a dense quadratic term involving the kernel matrix, making direct optimization challenging.
Our two-phase method combines the robust global convergence of first-order methods to navigate the early optimization landscape with the rapid local convergence of second-order approaches. Specifically, Phase I employs an inexact ADMM to rapidly generate a high-quality ``warm start'' solution. Phase II subsequently activates the preconditioned semismooth Newton ALM to refine this estimate, exploiting its fast local convergence properties to achieve high accuracy efficiently.


\SetKwInOut{Input}{Input}
\SetKwInOut{Output}{Output}
\begin{algorithm}
	\caption{Phase I: Inexact ADMM for solving  \eqref{dual-1}}\label{alg:admm}
	\Input{$K \in \mathbb{S}^{n}_{++}$,
		$y \in \mathbb{R}^n$,
		$\tau \in (0, 1)$,
		$\lambda > 0$,
		$\sigma > 0$,
        $\gamma \in (0, \frac{1+\sqrt{5}}{2})$, $\{\varepsilon_k\geq 0\}$ satisfying $\sum_{k=0}^{\infty} \varepsilon_k<\infty$,
        		$\beta^0 \in \mathbb{R}$,
        $z^0,v^0 \in \mathbb{R}^n$,
	}
	\Output{
the final iterate $ (\alpha^k, v^k, \beta^k, z^k)$.

    }
    Compute the coefficient matrix \( M = K + \lambda\sigma(I + \mathbf{1}_n\mathbf{1}_n^{\top}) \)\;
	\For{$k = 0, 1, 2, \ldots$}{
		 \( b_k = \lambda(y - \beta^k\mathbf{1}_n - z^k + \sigma v^k) \)\;
        Approximately solve \( M\alpha = b_k \) to find \( \alpha^{k+1} \) satisfying \( \|b_k - M\alpha^{k+1} \|_2 \leq \varepsilon_k \)\;
		$ v^{k+1} = \Pi_{\mathcal{B}}(\alpha^{k+1} + z^k/\sigma) $\;
		$ \beta^{k+1} = \beta^k + \gamma\sigma\mathbf{1}_n^{\top} \alpha^{k+1} $\;
		$ z^{k+1} = z^k + \gamma\sigma(\alpha^{k+1} - v^{k+1}) $.
	}
\end{algorithm}
	\subsection{Phase I: Warm-start via inexact ADMM}

    In the first phase, we employ an inexact ADMM to quickly compute a warm-start solution for the dual problem \eqref{dual-1}. Based on the augmented Lagrangian function in \eqref{lagrangianfunction}, the ADMM generates the primal-dual sequence $\{(\beta^k,z^k;\alpha^k,v^k)\}$ via the following block minimization and dual ascent steps:
    \begin{align*}
    	& \alpha^{k+1}  \approx
    	\operatornamewithlimits{argmin}_{\alpha} \, L_{\sigma}(\alpha,v^k;\beta^k,z^k),\\
        &  v^{k+1}  \approx
    	\operatornamewithlimits{argmin}_{v} \, L_{\sigma}(\alpha^{k+1},v;\beta^k,z^k),\\
    	&(\beta^{k+1},z^{k+1})  =  (\beta^k,z^k) + \gamma \sigma \nabla_{(\beta,z)} L_{\sigma}(\alpha^{k+1},v^{k+1};\beta^k,z^k),
    \end{align*}
    where $\gamma \in (0, \frac{1+\sqrt{5}}{2})$ is the step length. The  minimization with respect to $\alpha$   involves solving an $n\times n$ linear system. To maintain computational efficiency, we solve this system inexactly. Specifically, $\alpha^{k+1}$ is obtained by approximately solving
\begin{equation}\label{eq:admm_alpha_system}
\left[K+\lambda\sigma(I+{\bf 1}_n{\bf 1}_n^{\top})\right]\alpha \approx \lambda(y-\beta^k{\bf 1}_n-z^k+\sigma v^k).
\end{equation}
The preconditioning strategy developed in Section~\ref{sec:pre} can be used to solve this system efficiently.
Subsequently, the update for $v$  admits a closed-form solution via projection onto the box $\mathcal{B}=[\tau - 1,\tau]^n$:
\begin{equation*}
    v^{k+1} = \Pi_{\mathcal{B}}(\alpha^{k+1} + z^k/\sigma),
\end{equation*}
where $\Pi_{\mathcal{B}}(\cdot)$ denotes the Euclidean projection onto $\mathcal{B}$. The variables $\beta$ and $z$ can be updated using the gradient $\nabla_{(\beta,z)} L_{\sigma}(\alpha^{k+1},v^{k+1};\beta^k,z^k)=({\bf 1}_n^{\top}\alpha^{k+1}, \alpha^{k+1}-v^{k+1})$.
The complete procedure is summarized in Algorithm~\ref{alg:admm}. Its convergence properties are formally established in Theorem~\ref{thm:admm_convergence} below, which follows directly from \cite[Theorem 5.1]{chen2017efficient}.
\begin{theorem}\label{thm:admm_convergence}
    Let $\{(\alpha^k, v^k, \beta^k, z^k)\}$ be the sequence generated by Algorithm~\ref{alg:admm}. Then, the sequence $\{(\alpha^k, v^k)\}$ converges to the unique optimal solution of problem \eqref{dual-1}.
\end{theorem}

\subsection{Phase II: Refinement via semismooth Newton ALM}\label{sec:ALM}
	The second phase refines the warm-start solution  obtained from Phase I to achieve high accuracy using a semismooth Newton ALM. This phase follows the update rules \eqref{ALM-subproblem} and \eqref{ALM-multiplier}, and the complete procedure is given in Algorithm~\ref{alg:alm}.
At each iteration $k$, the core computational task is solving the  subproblem \eqref{ALM-subproblem}. We employ a semismooth Newton method in Line 2 for this purpose, denoted as the subroutine  ``{\tt SSN}''  (detailed in Algorithm~\ref{alg:ssnal}).  The global convergence of Algorithm~\ref{alg:alm} to an optimal solution and its local linear convergence rate are guaranteed under standard stopping criteria \cite{rockafellar1976augmented,rockafellar1976monotone}.  For completeness, the detailed convergence analysis and specific stopping criteria are provided in the Appendix~\ref{sec:prelim-con}.

	\SetKwInOut{Input}{Input}\SetKwInOut{Output}{Output}
	\begin{algorithm}
		\caption{Phase II: semismooth Newton ALM for solving \eqref{dual-1}}\label{alg:alm}
		\Input{$K\in\mathbb{S}^{n}_{++}$, $y\in\mathbb{R}^n$, $\tau\in(0,1)$,
$\lambda > 0$, $\sigma_0>0$,
$\beta^0\in\mathbb{R}$,
$\alpha^0,v^0,z^0\in\mathbb{R}^n$ (initialized from Phase I).
		}
		\Output{An approximate primal-dual solution  $ (\alpha^k, v^k, \beta^k,z^k)$ to problem \eqref{dual-1}.
        }
		\For{$k=0, 1, 2, \ldots$}{
			$(\alpha^{k+1},v^{k+1})=$
			{\tt SSN}$(K,y,\tau,\lambda,\beta^k,z^k,\sigma_k,\alpha^k)$\tcp*[l]{via Algorithm~\ref{alg:ssnal}}
			$ \beta^{k+1} = \beta^k + \sigma_k  {\bf 1}_n^{\top} \alpha^{k+1} $\;
			$z^{k+1} = z^k + \sigma_k  (\alpha^{k+1} - v^{k+1})$\;
			Update $\sigma_{k+1} \geq \sigma_k$.
		}
	\end{algorithm}

	The major computational challenge in Algorithm~\ref{alg:alm} lies in solving the ALM subproblem \eqref{ALM-subproblem} efficiently. To simplify this task, we observe that for fixed $\beta^k, z^k, \sigma_k$, the minimization with respect to $v$ admits a closed-form solution $v = \Pi_{\mathcal{B}}(\alpha + z^k/\sigma_k)$.  Substituting this back into the augmented Lagrangian reduces the subproblem to minimizing a continuously differentiable function of $\alpha$:
	\begin{equation}\label{def:phi_k}
	   \min_{\alpha\in\mathbb{R}^n} \phi_k(\alpha) := \frac{1}{2\lambda} \alpha^{\top} K \alpha - y^{\top}\alpha + \frac{\sigma_k}{2}({\bf 1}_n^{\top}\alpha + \beta^k/\sigma_k)^2 + \mathcal{M}^{\sigma_k}_{\delta_{\mathcal{B}}}(\alpha+z^k/\sigma_k),
	\end{equation}
where $\mathcal{M}^{\sigma}_{f}(\cdot)$ denotes the Moreau envelope of a function $f$; see the definition in Appendix~\ref{sec:prelim}.
	Solving this minimization problem is equivalent to finding $\alpha$ such that $\nabla \phi_k(\alpha)=0$. Using the gradient property of the Moreau envelope (see \cite{rockafellar2009variational} and \eqref{def-grad-MY} in the Appendix~\ref{sec:prelim}), 
the first-order optimality condition is:
    \begin{equation}\label{eq:optimality_condition}
		0 = \nabla\phi_k(\alpha) = \frac{1}{\lambda} K\alpha - y + \beta^k{\bf 1}_n + \sigma_k{\bf 1}_n{\bf 1}_n^{\top}\alpha + \sigma_k\left(\alpha + z^k/\sigma_k - \Pi_{\mathcal{B}}(\alpha+z^k/\sigma_k)\right).
	\end{equation}
	A key challenge in this approach is that equation \eqref{eq:optimality_condition} contains the nonsmooth projection operation $\Pi_{\mathcal{B}}(\cdot)$, making the classical Newton method inapplicable. To address this, we introduce the semismooth Newton method, a principled generalization of Newton's method \cite{kummer1988newton,qi1993nonsmooth,facchinei2007finite}. This approach is particularly suitable for our problem because $\Pi_{\mathcal{B}}(\cdot)$ is strongly semismooth. Consequently, the gradient $\nabla \phi_k$ is a strongly semismooth function (specifically a piecewise linear function), ensuring the rapid convergence of the Newton iterations.

	To implement the method, we first need to characterize the generalized derivative of the nonsmooth components. The Clarke generalized Jacobian of the piecewise linear function $\Pi_{\mathcal{B}}(\cdot)$ is given as follows:
	\begin{align*}
		&\partial \Pi_{\mathcal{B}}(v) =\partial \Pi_{[\tau-1,\tau]}(v_1)\times\cdots\times \partial \Pi_{[\tau-1,\tau]}(v_n),\,v\in\mathbb{R}^n, \\
		&\mbox{where } \partial \Pi_{[\tau-1,\tau]}(v_1) =
		\begin{cases}
			\{1\}, &  \mbox{ if }  \tau-1 <  v_1 < \tau, \\
			\{0\}, & \mbox{ if } v_1 < \tau-1 \mbox{ or } v_1 > \tau, \\
			[0,1], & \mbox{ if } v_1 = \tau -1 \mbox{ or } v_1 = \tau.
		\end{cases}
	\end{align*}
	Based on this, we can define the collection of generalized Hessians of the function $\phi_k$ as the following set-valued mapping:
	\begin{equation}\label{eq:hessian}
		\partial^2 \phi_k (\alpha) =
		\left\{ \frac{1}{\lambda} K + \sigma_k \left({\bf 1}_n{\bf 1}_n^{\top} + I - S \right) :  S\in\partial \Pi_{\mathcal{B}}(\alpha + z^k/\sigma_k)
		\right\}.
	\end{equation}
	Notably, since $S$ is a diagonal matrix with entries in $[0,1]$, the matrix $I - S$ is positive semidefinite. Therefore, given that $K$ is positive definite, any matrix in the set \eqref{eq:hessian} is positive definite, ensuring the well-posedness of the Newton linear systems.

	\begin{algorithm}[htbp]
		\caption{Semismooth Newton method for solving subproblem \eqref{ALM-subproblem}: $(\alpha,v)=$ {\tt SSN}$(K,y,\tau,\lambda,\beta,z,\sigma,\alpha^0)$
		}\label{alg:ssnal}
		\Input{$K\in\mathbb{S}^{n}_{++}$, $y\in\mathbb{R}^n$, $\tau\in(0,1)$, $\lambda > 0$,
		$\beta\in\mathbb{R}$, $z\in\mathbb{R}^n$, $\sigma>0$, $\alpha^0\in\mathbb{R}^n$,
        $\bar{\eta}\in(0,1)$,
        $\iota \in (0,1]$,
        $\mu \in (0,1/2)$,
        $r \in (0,1)$, $\tau_1,\tau_2\in(0,1)$.
        }
		\Output{An approximate solution $(\alpha^t,v^t)$ to  subproblem \eqref{ALM-subproblem}.}
		\For{$t=0, 1, 2,\ldots$}{
			Compute  $\nabla\phi_k(\alpha^{t})$ via \eqref{eq:optimality_condition} and choose  $H_t \in \partial^2 \phi_k (\alpha^{t})$ via \eqref{eq:hessian}\;
			Let $\varepsilon_t={\tau_1}\min\{\tau_2,\|\nabla \phi_k(\alpha^t)\|_2\}$\;
 Apply  preconditioned conjugate gradient (PCG) to find an approximate solution $d^{\,t}$ to
			\begin{equation}\label{newton-sys}
				(H_t + \varepsilon_t I) \, d = -\nabla\phi_k(\alpha^{t})
			\end{equation}
			satisfying $\| (H_t + \varepsilon_t I) \, d^{\,t}+\nabla\phi_k(\alpha^{t})\|_2 \leq \min(\bar{\eta},\|\nabla\phi_k(\alpha^{t})\|_2^{1+\iota})$\;
			Set $c_t =1$\;
			\While{$\phi_k(\alpha^{t} + c_t \, d^{\,t}) > \phi_k(\alpha^{t}) + \mu c_t \langle \nabla\phi_k(\alpha^{t}),d^{\,t} \rangle$}{
            $c_t = r c_t$\tcp*[l]{backtracking line search}}
			$ \alpha^{\,t+1} = \alpha^{t} + c_t \,d^{\,t}$\;
		}
		Compute $v^{t}:= \Pi_{\mathcal{B}}(\alpha^{t} + z/\sigma)$.
	\end{algorithm}

	The detailed steps of the semismooth Newton method for solving $\min \phi_k(\alpha)$ (i.e., finding a root $\nabla \phi_k(\alpha) = 0$ in \eqref{eq:optimality_condition}) are presented in Algorithm~\ref{alg:ssnal}. Like the classical Newton method for smooth equations, the semismooth Newton method with a unit step length only works locally near the solution. 
    To ensure global convergence, we adopt a standard backtracking line search strategy in 
    Lines 5--8 of Algorithm~\ref{alg:ssnal}. This strategy is well-founded because the direction $d^{\,t}$ computed in Line 4 is guaranteed to be a descent direction for the objective function $\phi_k$; for details, see \cite[Section 8.3.3]{facchinei2007finite}.

The convergence properties of Algorithm~\ref{alg:alm}, together with those of the semismooth Newton subroutine in Algorithm~\ref{alg:ssnal}, are summarized in Appendix~\ref{sec:prelim-con}.

	\section{An effective preconditioner for linear system \eqref{newton-sys}}\label{sec:pre}

One of the main computational bottlenecks of Algorithm~\ref{alg:ssnal} is solving the $n\times n$ linear system \eqref{newton-sys},
where $n$ is the number of data points and can be very large in practice. For convenience, we repeat the linear system as follows:
\begin{equation}\label{eqn-K}
\left( K + \lambda \sigma_k {\bf 1}_n{\bf 1}_n^{\top} + L \right) d = b, \quad L:=\lambda \sigma_k \left( I - S \right) + \lambda \varepsilon_t I.
\end{equation}
The eigenvalues of the kernel matrix $K$ tend to decay rapidly and  even exponentially (see \cite{williams2000effect,ma2017diving,altschuler2023kernel}), leading to severe ill-conditioning of the  linear system \eqref{eqn-K}. This issue is particularly severe in high dimensions, where the more pronounced eigenvalue decay leads to extreme ill-conditioning, significantly hindering the convergence speed and numerical stability of iterative solvers such as the CG method. These challenges are not unique to our setting; they are well known and pervasive in kernel-based methods, which involve the kernel matrix and therefore face significant scalability issues on large data. A direct Cholesky-based method for solving \eqref{eqn-K} requires $\mathcal{O}(n^3)$ operations, making it impractical when $n\geq 10^4$.

To enhance scalability, we incorporate preconditioning techniques into the CG method. We find a low-rank approximation $\widehat{K}$ of the kernel matrix $K$ and then define 
\begin{equation} \label{eq:preconditioner}
    P := \widehat{K} + \lambda \sigma_k {\bf 1}_n{\bf 1}_n^{\top} + L.
\end{equation}
While various low-rank approximation methods exist, the Randomly Pivoted Cholesky (RPCholesky) algorithm has emerged as a state-of-the-art choice for preconditioning due to its compelling balance of computational efficiency and approximation quality \cite{chen2023randomly, diaz2023robust}. We therefore employ this method. For a given positive semidefinite matrix $K \in \mathbb{S}^n_+$ and a target rank $r$, the algorithm, denoted by $F = \mathrm{RPCholesky}(K,r)$, produces a factor matrix $F \in \mathbb{R}^{n \times r}$ that yields the approximation $\widehat{K} = FF^\top \approx K$.

Building on the approximation \eqref{eq:preconditioner}, we use the preconditioner via the map $d \to P^{-1}d$ in the preconditioned CG (PCG) method for solving \eqref{eqn-K}. The following lemma describes how to compute $P^{-1}$ efficiently.

\begin{lemma}\label{lemma:preconditioner}
For $F\in\mathbb{R}^{n\times r}$,  nonsingular  $L\in\mathbb{R}^{n\times n}$, and  $\gamma>0$,
\begin{equation*}
(FF^{\top} + \gamma {\bfn}{\bfn}^{\top} + L)^{-1}
= L^{-1} - L^{-1} \!\begin{bmatrix}F & {\bfn}\end{bmatrix}
\!\left(
\begin{bmatrix}I_r & \\ & \gamma^{-1}\end{bmatrix}
+ \begin{bmatrix}F^{\top} \\ {\bfn}^{\top}\end{bmatrix}
L^{-1} \begin{bmatrix}F & {\bfn}\end{bmatrix}
\right)^{\!-1}
\!\begin{bmatrix}F^{\top} \\ {\bfn}^{\top}\end{bmatrix} L^{-1}.
\end{equation*}
\end{lemma}
\begin{proof}
First, note that
$
FF^{\top} + \gamma {\bfn}{\bfn}^{\top}
= \begin{bmatrix} F & {\bfn} \end{bmatrix}
\begin{bmatrix} I_r & \\ & \gamma \end{bmatrix}
\begin{bmatrix} F^{\top} \\ {\bfn}^{\top} \end{bmatrix}
$.
Applying the Sherman-Morrison-Woodbury formula,
we obtain the desired expression.
\end{proof}

	Now, we analyze the computational cost of applying the preconditioner $P^{-1}$ to a vector $d$. Since $L$ defined in \eqref{eqn-K} is  diagonal, computing $L^{-1}$ requires only $\mathcal{O}(n)$ operations.
Next, forming the $(r+1)\times (r+1)$ matrix
$ \begin{bmatrix}
		I_r&\\&\frac{1}{\lambda\sigma_k}
	\end{bmatrix}+\begin{bmatrix}F^{\top}\\{\bfn}^{\top}\end{bmatrix} L^{-1}\begin{bmatrix}F&{\bfn}\end{bmatrix}$
requires $\mathcal{O}(r^2n)$ operations, while inverting this matrix costs $\mathcal{O}(r^3)$.
Finally, computing $P^{-1}d$ involves a sequence of matrix-vector multiplications using Lemma~\ref{lemma:preconditioner}. Therefore, the overall computational cost of applying  $P^{-1}$ to $d$ is $\mathcal{O}(r^3+r^2n)\approx \mathcal{O}(r^2n)$. 
The preconditioner can also be employed in Phase I for solving the linear system \eqref{eq:admm_alpha_system}.

We now analyze the convergence properties of the PCG method with the preconditioner $P$ defined in \eqref{eq:preconditioner}. Denote the coefficient matrix of the linear system \eqref{eqn-K} as
\begin{equation}\label{eq-A}
A:= K + \lambda \sigma_k {\bf 1}_n{\bf 1}_n^{\top} + L.
\end{equation}
We study the symmetrically preconditioned matrix $P^{-1/2}AP^{-1/2}$ for the theoretical analysis rather than the nonsymmetric matrix $P^{-1}A$ used in our practical left-preconditioned CG implementation. This symmetric form facilitates convergence analysis while being mathematically equivalent in terms of the CG iteration sequence \cite{saad2003iterative}.
Our analysis follows \cite[Theorem~2.2]{diaz2023robust}, but extends the coefficient matrix from the case $K + \gamma I$ to the more general form given by \eqref{eq-A}.

	Next, we present the main performance guarantee for the PCG method equipped with the RPCholesky preconditioner $P$. To this end, we first recall in Definition~\ref{def-tail} the notion of $\mu$-tail rank for a positive definite matrix \cite[Definition~2.1]{diaz2023robust}, which quantifies eigenvalue decay and plays a central role in Theorem~\ref{thm:pcg rate}. The proof, adapted from \cite[Theorem~2.2]{diaz2023robust}, is provided in Appendix~\ref{app:thm3proof}.

\begin{definition}\label{def-tail}
    Let \( A \in \mathbb{S}^n_{++} \) be a positive definite matrix with eigenvalues \( \lambda_1(A) \geq \lambda_2(A) \geq \cdots \geq \lambda_n(A) > 0 \). The \( \mu \)-tail rank of \( A \) is defined as
    \[
    \mathrm{rank}_\mu(A) := \min \left\{ t \in \mathbb{N} : \sum_{i = t+1}^n \lambda_i(A) \leq \mu \right\}.
    \]
\end{definition}

\begin{theorem}\label{thm:pcg rate}
		Let $K$ be a positive semidefinite matrix, $\delta \in (0,1)$, and $\varepsilon \in (0,1)$ be an error tolerance. Construct a random low-rank approximation $\widehat{K}$ of $K$ using the RPCholesky algorithm with approximation rank $r$ satisfying
		\begin{equation*}
			r \geq \min\Biggl\{n,{\rm rank}_{\mu}(K)\Bigl(1+\log\frac{\mathrm{tr} (K)}{{\rm tr}(K-\lfloor K\rfloor_{{\rm rank}_{\mu}(K)})}\Bigr)\Biggr\}. 
		\end{equation*}
Let
$$
L \succeq \mu I, \quad P= \widehat{K} + \lambda \sigma_k {\bf 1}_n{\bf 1}_n^{\top} + L, \quad A=K + \lambda \sigma_k {\bf 1}_n{\bf 1}_n^{\top} + L.
$$
Then, with probability at least $1-\delta$, the preconditioned matrix satisfies
		\begin{equation}\label{bound:cond_no}
			\kappa\Bigl(P^{-1/2} A P^{-1/2}\Bigr)\leq {3}/{\delta}.
		\end{equation}
		Furthermore, on the event that the bound \eqref{bound:cond_no} holds, when applying the PCG method with preconditioner $P$ to solve the linear system $Ad=b$, the iterate $d^{\,t}$ satisfies
		\begin{equation*}
			\|d^{\,t}-d^*\|_{A}\leq \varepsilon\|d^*\|_{A}
		\end{equation*}
		at any iteration $t\geq
\delta^{-1/2}\log({2}/{\varepsilon})$,
where $d^*$ denotes the exact solution.
	\end{theorem}
The RPCholesky algorithm constructs the preconditioner $P$ by sampling $r$ columns of $K$ without replacement. Theorem~\ref{thm:pcg rate} shows that this preconditioning strategy is particularly effective when $K$ has low numerical rank. In particular, if ${\rm rank}_\mu(K) = \mathcal{O}(\sqrt{n})$, then choosing $r=\mathcal{O}(\sqrt{n})$ is sufficient to bound the condition number of the preconditioned system. This bound is independent of $n$ for fixed $\delta$. Consequently, the preconditioned conjugate gradient (PCG) method reaches a prescribed tolerance $\varepsilon$ in $\mathcal{O}(\log(1/\varepsilon))$ iterations, yielding an overall complexity of $\mathcal{O}(n^2)$ for solving the linear system \eqref{eqn-K}.

    While the theoretical analysis provides guidance for choosing the estimation rank $r$, implementing this choice in large-scale settings remains challenging because computing ${\rm rank}_\mu(K)$ requires spectral information on $K$. In addition, the appropriate value of $r$ may vary with the lower-bound parameter $\mu$ in $L \succeq \mu I$, which can change across algorithm iterations. In our setting, one may take $\mu=\lambda\varepsilon_t$. Alternatively, choosing $r$ as large as possible may reduce the number of PCG iterations, but it also increases the per-iteration cost and can therefore be computationally inefficient overall.
    To address this issue, we adopt a heuristic strategy for selecting $r$. We first set $r_0=\sqrt{n}$ and compute an initial approximation $\widehat{K}_0=F_0F_0^{\top}$ using the RPCholesky algorithm. We then compute the eigenvalues $\lambda_1\geq\dots\geq\lambda_{r_0}$ of $\widehat{K}_0$ and choose a threshold $\xi\in(0,1)$. Since $F_0^{\top}F_0$ and $\widehat{K}_0$ share the same nonzero eigenvalues, this spectral computation costs only $\mathcal{O}(n^{3/2})$, rather than $\mathcal{O}(n^3)$. The final rank $r$ is taken to be the smallest integer satisfying $\xi\lambda_1\geq\lambda_r$, after which we rerun RPCholesky to construct $\widehat{K}$. Importantly, $\widehat{K}$ is computed only once throughout the entire algorithm.
    This strategy selects $r$ at a moderate computational cost and improves the overall algorithmic performance, as shown in the numerical experiments. Rather than relying on the eigenvalues of $K$ directly, it determines $r$ adaptively from the eigenvalues of $\widehat{K}_0$.
	It is worth noting that even for matrices lacking a clear low-rank structure, our preconditioning method remains effective. Although alternative approaches for handling such cases have been proposed in the literature \cite{zhao2024adaptive}, our current method offers a reasonable balance between theoretical guarantees and practical efficiency for the problems considered in this paper, even if it is not necessarily optimal. A comprehensive exploration of adaptive techniques for varying problem structures remains an important direction for future research.

\section{Numerical experiments}\label{sec:num}
In this section, we evaluate the computational performance of PALM-KQR for solving the KQR dual problem \eqref{dual-1} and compare it with benchmark solvers on synthetic and real-world datasets. We measure the quality of an approximate primal-dual solution $(\alpha, v, \beta, z)$ using the relative Karush-Kuhn-Tucker (KKT) residual
	\begin{equation*}
		\eta_{\rm kkt} = \max\{\eta_p,\eta_d,\eta_c\},
	\end{equation*}
where the primal feasibility ($\eta_p$), dual feasibility ($\eta_d$), and complementarity ($\eta_c$) residuals are given by
\begin{equation*}
    \eta_p = \frac{\|z - y  + \beta {\bf 1}_n + \frac{1}{\lambda} K\alpha\|_2}{1+\|y\|_2}, \quad
    \eta_d = \frac{\sqrt{({\bf 1}_n^{\top} \alpha)^2 +\|\alpha - v\|_2^2} }{1+\|\alpha\|_2}, \quad
    \eta_c = \frac{\|v - \Pi_{\mathcal{B}}(z + v)\|_2}{1+\|v\|_2}.
\end{equation*}
Additionally, we compute the relative duality gap
	\begin{equation*}
\eta_{\rm gap} = \frac{|{\rm obj}_P-{\rm obj}_D|}{1+|{\rm obj}_P| + |{\rm obj}_D|},\quad
\text{where} \quad
{\rm obj}_P = \rho_{\tau}(z)+ \frac{1}{2\lambda}\alpha^{\top}K\alpha, \quad
{\rm obj}_D = - \frac{1}{2\lambda}\alpha^{\top}K\alpha + y^{\top}\alpha.
	\end{equation*}
Let $\epsilon > 0$ be a given accuracy tolerance. Algorithm~1 (Phase~I warm-start) is terminated when $\max\{\eta_{\rm kkt}, \eta_{\rm gap}\} \leq 10^{-3}$ or when the iteration count reaches 100. Algorithm~2 (Phase~II refinement) is terminated when $\max\{\eta_{\rm kkt}, \eta_{\rm gap}\} \leq \epsilon$ or when the iteration count reaches 1000.

We compare  the performance of our algorithm against two benchmark solvers.
The first is Gurobi (academic license, version 12.00), a state-of-the-art commercial solver for convex QP. We configure Gurobi's barrier interior-point method to solve the convex QP formulation \eqref{dual-0} with a convergence tolerance matching our target accuracy $\epsilon$.
The second benchmark is fastKQR \cite{tang2024fastkqr}, a specialized R package implementing a finite smoothing algorithm. We run fastKQR using its standard stopping criterion  described in \cite{tang2024fastkqr} with a tolerance of $\epsilon$ and a maximum iteration limit of $10^6$.

All experiments were conducted using MATLAB (version 9.12) on a Linux workstation equipped with an Intel Xeon Platinum 8375C CPU (2.90 GHz, 128 cores) and 1024 GB of RAM.


\subsection{Synthetic data}\label{sec:syn data}
We first evaluate the performance of PALM-KQR, fastKQR, and Gurobi under various parameter settings and problem dimensions using synthetic data. Additionally, to explicitly assess the impact of our proposed preconditioning strategy, we evaluate NALM-KQR, a baseline variant of our framework that solves the inner linear systems using the standard CG method without preconditioning.

We adopt the data generation procedure described in \cite{yuan2006gacv}. Specifically, let the input vector be $x_i = (x_{i,1}, x_{i,2})^\top \in \mathbb{R}^2$, where the components $x_{i,1}$ and $x_{i,2}$ are independently drawn from the uniform distribution on the interval $[0,1]$. The corresponding response $y_i$ is generated as follows:
\begin{equation*}
    y_i = \frac{40\exp\left\{8((x_{i,1}-0.5)^2+(x_{i,2}-0.5)^2)\right\}}{\exp\left\{8((x_{i,1}-0.2)^2+(x_{i,2}-0.7)^2)\right\} + \exp\left\{8((x_{i,1}-0.7)^2+(x_{i,2}-0.2)^2)\right\}} + \varepsilon,
\end{equation*}
where the noise term $\varepsilon$ is drawn from the standard normal distribution.

In these experiments, we employ the radial basis kernel $k(x,x') = \exp(-\gamma \|x-x'\|_2^2)$ and vary the bandwidth parameter $\gamma$, the sample size $n$, and the quantile level $\tau$ across the following sets:
\[
    \gamma \in \{10^{-1}, 10^{-2}, 10^{-3}\},\quad
    n \in \{5\times 10^3, \, 10^4, \, 2\times 10^4\}, \quad
    \tau \in \{0.1, \, 0.5, \, 0.9\}.
\]
For the regularization parameter $\lambda$, we consider a sequence of 50 values logarithmically spaced in the interval $[10^{-2}, 10^0]$. We apply all solvers to the KQR problems corresponding to this full sequence of $\lambda$ values and record the total computation time.  The convergence tolerance is set to $\epsilon = 10^{-8}$.

\begin{table}[!ht]
    \centering
    \caption{Performance comparison of PALM-KQR, NALM-KQR, Gurobi, and fastKQR on synthetic data using the radial basis kernel. The time is reported in ``hours:minutes:seconds''. The column \# indicates the number of instances (out of 50) solved successfully within the 2-hour time limit.}
    \label{tab:Rbf_combined}
    \begin{tabular}{cccrcrcrcrc}
        \toprule
        \multirow{2}{*}{$\gamma$} & \multirow{2}{*}{$n$} & \multirow{2}{*}{$\tau$} & \multicolumn{2}{c}{PALM-KQR} & \multicolumn{2}{c}{NALM-KQR} & \multicolumn{2}{c}{Gurobi} & \multicolumn{2}{c}{fastKQR}\\
         \cmidrule(lr){4-5} \cmidrule(lr){6-7} \cmidrule(lr){8-9} \cmidrule(l){10-11}
        &  &  & Time & \# & Time & \# & Time & \# & Time & \# \\
        \midrule
        \multirow{9}{*}{$10^{-1}$}
        & \multirow{3}{*}{$5\times 10^3$}
          & 0.1 & 45 & 50 & 4:10 & 50 & 4:08 & 50 & 3:17 & 50 \\
        & & 0.5 & 34 & 50 & 4:21 & 50 & 4:12 & 50 & 3:21 & 50 \\
        & & 0.9 & 36 & 50 & 4:14 & 50 & 4:12 & 50 & 3:20 & 50 \\
        \cmidrule{2-11}
        & \multirow{3}{*}{$10^4$}
          & 0.1 & 2:26 & 50 & 51:23 & 50 & 17:50 & 50 & 15:14 & 50 \\
        & & 0.5 & 2:32 & 50 & 50:45 & 50 & 17:50 & 50 & 15:14 & 50 \\
        & & 0.9 & 2:41 & 50 & 50:55 & 50 & 17:50 & 50 & 15:36 & 50 \\
        \cmidrule{2-11}
        & \multirow{3}{*}{$2\times 10^4$}
          & 0.1 & 14:41 & 50 & 2:30:23 & 5 & 1:17:37 & 50 & 1:01:34 & 50 \\
        & & 0.5 & 14:23 & 50 & 2:12:55 & 9 & 1:17:32 & 50 & 59:09 & 50 \\
        & & 0.9 & 13:21 & 50 & 2:45:09 & 6 & 1:17:32 & 50 & 1:01:54 & 50 \\
        \midrule
        \multirow{9}{*}{$10^{-2}$}
        & \multirow{3}{*}{$5\times 10^3$}
          & 0.1 & 40 & 50 & 4:05 & 50 & 4:12 & 50 & 3:09 & 50 \\
        & & 0.5 & 35 & 50 & 4:12 & 50 & 4:10 & 50 & 3:12 & 50 \\
        & & 0.9 & 35 & 50 & 4:14 & 50 & 4:12 & 50 & 3:14 & 50 \\
        \cmidrule{2-11}
        & \multirow{3}{*}{$10^4$}
          & 0.1 & 2:01 & 50 & 50:25 & 50 & 17:50 & 50 & 15:01 & 50 \\
        & & 0.5 & 2:12 & 50 & 50:23 & 50 & 17:50 & 50 & 14:44 & 50 \\
        & & 0.9 & 2:24 & 50 & 50:41 & 50 & 17:52 & 50 & 14:36 & 50 \\
        \cmidrule{2-11}
        & \multirow{3}{*}{$2\times 10^4$}
          & 0.1 & 12:43 & 50 & 2:31:34 & 10 & 1:17:35 & 50 & 56:54 & 50 \\
        & & 0.5 & 12:43 & 50 & 2:14:41 & 11 & 1:17:32 & 50 & 57:09 & 50 \\
        & & 0.9 & 12:45 & 50 & 2:23:09 & 9 & 1:17:33 & 50 & 56:12 & 50 \\
        \midrule
        \multirow{9}{*}{$10^{-3}$}
        & \multirow{3}{*}{$5\times 10^3$}
          & 0.1 & 28 & 50 & 4:10 & 50 & 4:08 & 50 & 3:07 & 50 \\
        & & 0.5 & 29 & 50 & 4:21 & 50 & 4:12 & 50 & 3:03 & 50 \\
        & & 0.9 & 31 & 50 & 4:14 & 50 & 4:12 & 50 & 3:05 & 50 \\
        \cmidrule{2-11}
        & \multirow{3}{*}{$10^4$}
          & 0.1 & 1:23 & 50 & 45:31 & 50 & 17:50 & 50 & 15:14 & 50 \\
        & & 0.5 & 1:37 & 50 & 44:25 & 50 & 17:50 & 50 & 15:14 & 50 \\
        & & 0.9 & 1:51 & 50 & 44:28 & 50 & 17:50 & 50 & 15:36 & 50 \\
        \cmidrule{2-11}
        & \multirow{3}{*}{$2\times 10^4$}
          & 0.1 & 8:21 & 50 & 2:11:33 & 15 & 1:17:32 & 50 & 54:27 & 50 \\
        & & 0.5 & 8:32 & 50 & 2:12:51 & 16 & 1:17:33 & 50 & 55:36 & 50 \\
        & & 0.9 & 8:34 & 50 & 2:17:05 & 17 & 1:17:32 & 50 & 53:49 & 50 \\
        \bottomrule
    \end{tabular}
\end{table}
Table~\ref{tab:Rbf_combined} shows that PALM-KQR consistently outperforms NALM-KQR, Gurobi, and fastKQR on the synthetic problems with the radial basis kernel. For $n=5\times10^3$, PALM-KQR solves all 50 instances in 28--45 seconds, whereas Gurobi and fastKQR require about 3--4 minutes. This corresponds to a speedup of roughly 4--8 times. Its advantage becomes more pronounced as the problem size increases. For $n=2\times10^4$, PALM-KQR solves all 50 instances in about 8--15 minutes, while Gurobi requires about 1 hour 17 minutes and fastKQR about 54--62 minutes. This represents a substantial speedup of approximately 5--6 times over Gurobi and 4--5 times over fastKQR. These results demonstrate both the efficiency and the robustness of PALM-KQR across the tested RBF settings.

\begin{table}[!ht]
    \centering
    \caption{Performance comparison of PALM-KQR, NALM-KQR, Gurobi, and fastKQR on synthetic data using linear ($k(x,x')=x^{\top}x'$) and Laplacian ($k(x,x') = \exp(-0.1 \|x-x'\|_1)$) kernels.  The time is reported in ``hours:minutes:seconds''. The column \# indicates the number of instances (out of 50) solved successfully within the 2-hour time limit.}
    \label{tab:Linear_Laplacian_combined}
    \begin{tabular}{cccrcrcrcrc}
        \toprule
        \multirow{2}{*}{Kernel} & \multirow{2}{*}{$n$} & \multirow{2}{*}{$\tau$} & \multicolumn{2}{c}{PALM-KQR} & \multicolumn{2}{c}{NALM-KQR} & \multicolumn{2}{c}{Gurobi} & \multicolumn{2}{c}{fastKQR}\\
        \cmidrule(lr){4-5} \cmidrule(lr){6-7} \cmidrule(lr){8-9} \cmidrule(l){10-11}
        & & & Time & \# & Time & \# & Time & \# & Time & \# \\
        \midrule
        \multirow{9}{*}{Linear}
        & \multirow{3}{*}{$5\times 10^3$}
          & 0.1 & 10 & 50 & 21 & 50 & 23 & 50 & 17 & 50 \\
        & & 0.5 & 12 & 50 & 28 & 50 & 23 & 50 & 18 & 50 \\
        & & 0.9 & 15 & 50 & 31 & 50 & 24 & 50 & 18 & 50 \\
        \cmidrule{2-11}
        & \multirow{3}{*}{$10^4$}
          & 0.1 & 52 & 50 & 2:12 & 50 & 2:07 & 50 & 1:45 & 50 \\
        & & 0.5 & 54 & 50 & 2:26 & 50 & 2:07 & 50 & 1:48 & 50 \\
        & & 0.9 & 55 & 50 & 2:23 & 50 & 2:07 & 50 & 1:49 & 50 \\
        \cmidrule{2-11}
        & \multirow{3}{*}{$2\times 10^4$}
          & 0.1 & 4:17 & 50 & 9:16 & 50 & 9:01 & 50 & 8:23 & 50 \\
        & & 0.5 & 4:27 & 50 & 9:54 & 50 & 9:01 & 50 & 8:47 & 50 \\
        & & 0.9 & 4:26 & 50 & 10:02 & 50 & 9:01 & 50 & 8:44 & 50 \\
        \midrule
        \multirow{9}{*}{Laplacian}
        & \multirow{3}{*}{$5\times 10^3$}
          & 0.1 & 2:19 & 50 & 12:12 & 50 & 10:34 & 50 & 9:27 & 50 \\
        & & 0.5 & 2:23 & 50 & 12:45 & 50 & 10:34 & 50 & 9:28 & 50 \\
        & & 0.9 & 2:21 & 50 & 13:08 & 50 & 10:34 & 50 & 9:29 & 50 \\
        \cmidrule{2-11}
        & \multirow{3}{*}{$10^4$}
          & 0.1 & 5:16 & 50 & 50:34 & 50 & 24:35 & 50 & 24:01 & 50 \\
        & & 0.5 & 5:23 & 50 & 50:45 & 50 & 24:36 & 50 & 24:11 & 50 \\
        & & 0.9 & 5:22 & 50 & 50:45 & 50 & 24:36 & 50 & 24:34 & 50 \\
        \cmidrule{2-11}
        & \multirow{3}{*}{$2\times 10^4$}
          & 0.1 & 31:47 & 50 & 2:07:35 & 2 & 2:17:37 & 46 & 2:01:34 & 50 \\
        & & 0.5 & 31:39 & 50 & 2:12:17 & 2 & 2:17:37 & 46 & 1:59:09 & 50 \\
        & & 0.9 & 31:38 & 50 & 2:15:29 & 2 & 2:17:32 & 46 & 1:58:23 & 50 \\
        \bottomrule
    \end{tabular}
\end{table}

Table~\ref{tab:Rbf_combined} also shows that the runtime of PALM-KQR decreases as $\gamma$ decreases from $10^{-1}$ to $10^{-3}$. For example, when $n=2\times10^4$, the runtime drops from about 14 minutes to about 8 minutes. This indicates that PALM-KQR effectively leverages the low-rank structure of the kernel matrix to reduce runtime. In contrast, the runtimes of Gurobi and fastKQR are much less sensitive to $\gamma$, indicating that these solvers do not exploit the spectral properties of the kernel matrix effectively.

Finally, a comparison between PALM-KQR and its non-preconditioned variant, NALM-KQR, highlights the critical role of the proposed preconditioner. For smaller-scale problems ($n=5\times10^3$), NALM-KQR remains competitive, with runtimes comparable to Gurobi and fastKQR (around 3-4 minutes). However, its performance degrades severely as the problem dimension increases. At $n=2\times10^4$, NALM-KQR failed to solve the majority of the 50 instances within the 2-hour time limit (typically solving only 5--17 instances), rendering it far less efficient than the commercial benchmarks. This sharply contrasts with PALM-KQR, which consistently solved all instances well within the time limit. These results demonstrate that the proposed Nystr\"{o}m-based preconditioner is essential for ensuring both robustness and efficiency of the PALM-KQR framework.

We further report results for the linear and Laplacian kernels in Table~\ref{tab:Linear_Laplacian_combined}. The linear-kernel setting is less demanding than the nonlinear-kernel settings, and all solvers are correspondingly faster. Even so, PALM-KQR still achieves the best performance. For $n=2\times10^4$, PALM-KQR requires about 4 minutes, compared with about 8--9 minutes for Gurobi and fastKQR and about 9--10 minutes for NALM-KQR.
The Laplacian-kernel results present a more challenging setting. At $n=2\times10^4$, PALM-KQR solves all 50 instances in about 31 minutes, whereas Gurobi and fastKQR require about 2 hours, with Gurobi failing to solve all instances within the 2-hour limit. The non-preconditioned NALM-KQR performs much worse in this regime, solving only 2 of the 50 instances within the time limit. These results further confirm the importance of the proposed preconditioner in difficult large-scale settings.

\begin{figure}[htbp]
    \centering
    \begin{subfigure}[b]{0.48\textwidth}
        \includegraphics[width=\textwidth]{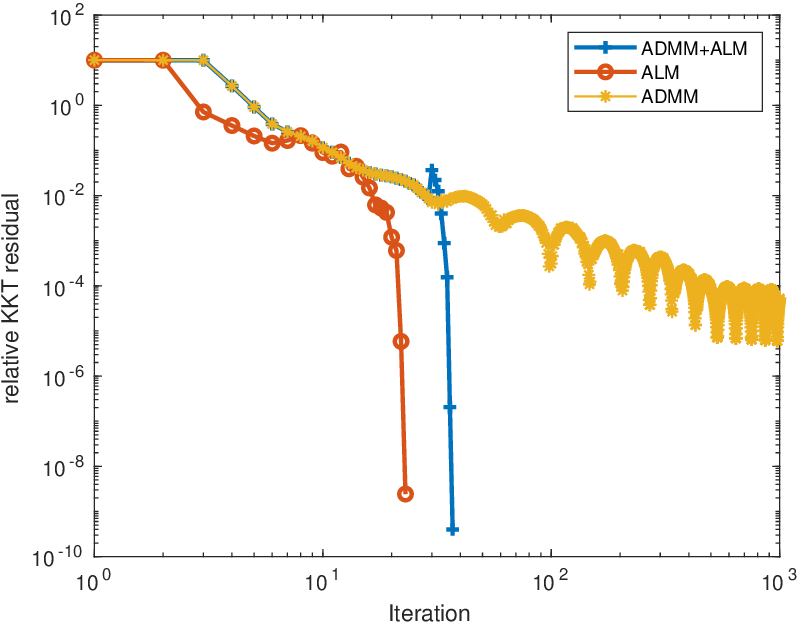}
    \end{subfigure}
    \hfill
    \begin{subfigure}[b]{0.48\textwidth}
        \includegraphics[width=\textwidth]{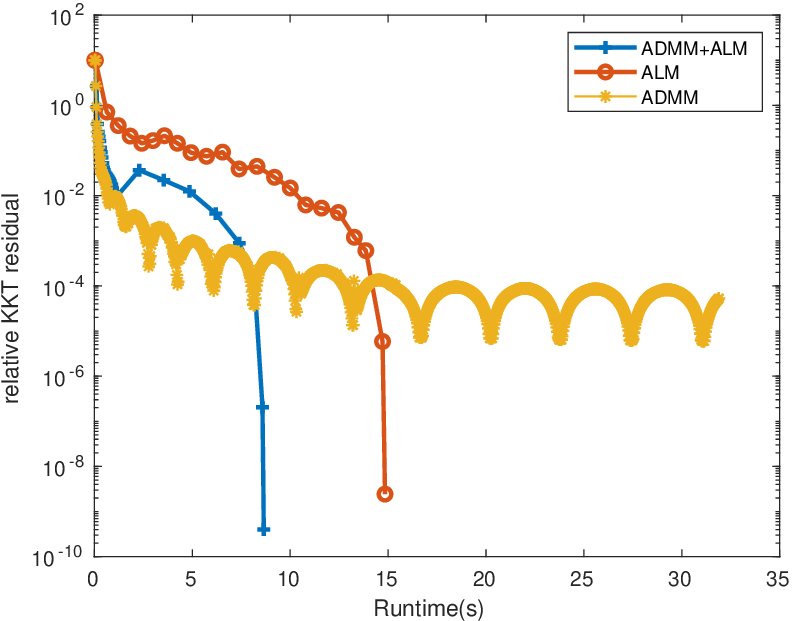}
    \end{subfigure}
    \caption{Convergence profiles of PALM-KQR (labeled as ADMM+ALM) versus its standalone components: ADMM (Phase~I) and semismooth Newton ALM (Phase~II). The plots display the relative KKT residual against iteration count (Left) and computation time (Right).}
    \label{fig:performance_comparison}
\end{figure}
We now analyze the internal mechanism of the PALM-KQR framework to validate its two-phase design. Figure~\ref{fig:performance_comparison} contrasts the performance of the full PALM-KQR framework against its standalone constituent phases, the inexact ADMM (Phase I) and the semismooth Newton ALM (Phase II), under identical experimental conditions. For this evaluation, we fix the problem size at $n=10^{4}$ and set  $\tau=0.1$, $\gamma=10^{-3}$,  $\lambda=10$, and $\epsilon=10^{-8}$. The maximum iteration limit is set to 1000 for both the standalone ADMM and semismooth Newton ALM, with all other internal settings matching their respective configurations within PALM-KQR.
The results illustrated in Figure~\ref{fig:performance_comparison} reveal the complementary convergence characteristics of the two methods. The standalone ADMM exhibits rapid initial progress, quickly reducing the error in the early stages, but suffers from slow ``tail convergence'', effectively plateauing before reaching the required tolerance. Conversely, the standalone semismooth Newton ALM (initialized from scratch) demonstrates slower initial progress but achieves rapid acceleration once it enters the local neighborhood of the solution. PALM-KQR effectively synthesizes these strengths: it utilizes ADMM to rapidly navigate to the solution's vicinity (the domain of attraction), and then switches to the semismooth Newton ALM to exploit its superlinear local convergence. Consequently, although the hybrid approach may incur a higher total iteration count (due to the accumulated ADMM steps), it achieves the target accuracy with significantly lower total computation time than using the semismooth Newton ALM alone.

\begin{table}[!htbp]
    \centering
    \caption{Computation time comparison of PALM-KQR, Gurobi, and fastKQR on real-world energy forecasting data for Switzerland ($n=8760$) and Germany ($n=8759$). The computation time is reported in ``hours:minutes:seconds''.}
    \label{tab:RealData_Combined}
    \begin{tabular}{ccccccccc}
        \toprule
        \multirow{2}{*}{$\gamma$} & \multirow{2}{*}{Year} & \multirow{2}{*}{$\tau$} & \multicolumn{3}{c}{Switzerland} & \multicolumn{3}{c}{Germany} \\
         \cmidrule(lr){4-6} \cmidrule(l){7-9}
         &  &  & \multicolumn{1}{c}{PALM-KQR} & \multicolumn{1}{c}{Gurobi} & \multicolumn{1}{c}{fastKQR} & \multicolumn{1}{c}{PALM-KQR} & \multicolumn{1}{c}{Gurobi} & \multicolumn{1}{c}{fastKQR} \\
        \midrule
        \multirow{6}{*}{$10^{-1}$}
        & \multirow{3}{*}{2021}
          & 0.1 & 11:27 & 1:15:00 & 1:02:25 & 10:07 & 1:02:14 & 55:14 \\
        & & 0.5 & 10:58 & 1:15:00 & 1:02:31 & 10:11 & 1:02:14 & 55:14 \\
        & & 0.9 & 11:34 & 1:15:00 & 1:02:32 & 10:16 & 1:02:14 & 55:23 \\
        \cmidrule{2-9}
        & \multirow{3}{*}{2022}
          & 0.1 & 9:07 & 58:43 & 51:17 & 10:21 & 1:07:11 & 57:28 \\
        & & 0.5 & 9:10 & 58:43 & 51:23 & 10:27 & 1:07:11 & 57:28 \\
        & & 0.9 & 9:07 & 58:43 & 51:29 & 10:17 & 1:07:11 & 57:34 \\
        \midrule
        \multirow{6}{*}{$10^{-2}$}
        & \multirow{3}{*}{2021}
          & 0.1 & 10:12 & 1:15:00 & 1:01:22 & 9:44 & 1:02:14 & 52:33 \\
        & & 0.5 & 10:11 & 1:15:00 & 1:01:23 & 9:45 & 1:02:14 & 52:14 \\
        & & 0.9 & 10:17 & 1:15:00 & 59:56 & 9:51 & 1:02:14 & 52:51 \\
        \cmidrule{2-9}
        & \multirow{3}{*}{2022}
          & 0.1 & 8:14 & 58:43 & 49:58 & 9:17 & 1:07:11 & 54:27 \\
        & & 0.5 & 8:25 & 58:43 & 49:46 & 9:21 & 1:07:11 & 54:21 \\
        & & 0.9 & 8:26 & 58:43 & 49:29 & 9:17 & 1:07:11 & 54:55 \\
        \midrule
        \multirow{6}{*}{$10^{-3}$}
        & \multirow{3}{*}{2021}
          & 0.1 & 7:35 & 1:15:01 & 53:23 & 7:17 & 1:02:14 & 43:56 \\
        & & 0.5 & 7:32 & 1:15:00 & 53:38 & 7:19 & 1:02:14 & 43:47 \\
        & & 0.9 & 7:38 & 1:15:00 & 53:31 & 7:23 & 1:02:14 & 43:29 \\
        \cmidrule{2-9}
        & \multirow{3}{*}{2022}
          & 0.1 & 6:49 & 58:43 & 42:21 & 7:28 & 1:07:11 & 43:33 \\
        & & 0.5 & 7:01 & 58:43 & 42:23 & 7:28 & 1:07:11 & 43:36 \\
        & & 0.9 & 6:54 & 58:43 & 41:48 & 7:19 & 1:07:11 & 43:38 \\
        \bottomrule
    \end{tabular}
\end{table}

\subsection{Energy forecasting data}
We further evaluate the performance of PALM-KQR, fastKQR, and Gurobi using a real-world energy forecasting dataset sourced from the Energy Charts platform. Accurate energy demand forecasting is critical for grid stability but remains challenging due to complex nonlinear dependencies on weather and calendar effects. To ensure a consistent and rigorous comparison with existing benchmarks, we adopt the preprocessed hourly electricity load data and the specific variable selection strategy validated in \cite{pernigo2025probabilistic}. The feature set comprises temperature, wind speed, hour of day, month, holiday indicators, and day of week, resulting in a feature dimension of $p=6$.

Regarding the experimental configuration, we retain the identical settings used in Section~\ref{sec:syn data}. We employ the radial basis kernel with $\gamma \in \{10^{-1}, 10^{-2}, 10^{-3}\}$ and test across quantile levels $\tau \in \{0.1, 0.5, 0.9\}$. For each case, we solve the KQR problem over a regularization path of 50 $\lambda$ values logarithmically spaced in $[10^0, 10^2]$, with the convergence tolerance fixed at $\epsilon = 10^{-8}$.

Table~\ref{tab:RealData_Combined} details the total computation times for PALM-KQR, Gurobi, and fastKQR on the real-world energy forecasting instances. Notably, all three algorithms successfully solved every one of the 50 instances to the target tolerance of $\epsilon=10^{-8}$.
The data  demonstrates the superior computational efficiency of PALM-KQR. Across all instances, the total runtime of PALM-KQR consistently remained in the range of approximately 6 to 11 minutes. In sharp contrast, the commercial solver Gurobi required between 58 minutes and 1 hour 15 minutes, while fastKQR's runtime varied from 41 minutes to 1 hour 2 minutes. Specifically, PALM-KQR is approximately 4 to 10 times faster than fastKQR and an even more impressive 5 to 12 times faster than Gurobi.
We also observe that PALM-KQR's computation time exhibits a  downward trend as $\gamma$ decreases. As $\gamma$ is reduced from $10^{-1}$ to $10^{-3}$, the average runtime for PALM-KQR drops from around 10--11 minutes to 7--8 minutes. This suggests that our algorithm effectively leverages the numerical low-rank structure of the kernel matrix, which becomes more pronounced at smaller bandwidths. Conversely, the choice of the quantile level $\tau$ had a negligible impact on the runtime for all algorithms, indicating their stability with respect to this parameter.

\section{Conclusion}\label{sec:conclusion}
In this paper, we proposed PALM-KQR,  a novel and highly efficient two-phase optimization framework designed to address the significant computational challenges of large-scale KQR. Our approach strategically combines a fast inexact ADMM to warm-start a high-accuracy semismooth Newton ALM for fast local convergence. A key innovation of our work is the development of a specialized preconditioning strategy, leveraging low-rank approximations of the kernel matrix. This strategy effectively mitigates the severe ill-conditioning of the linear systems arising in the Newton steps, which has been a major bottleneck in large-scale KQR computation.
Extensive numerical experiments on both synthetic and real-world energy forecasting datasets demonstrate that PALM-KQR  achieves state-of-the-art performance. The proposed algorithm exhibits substantial improvements in both computational speed and robustness compared to  leading commercial solvers (Gurobi) and specialized KQR packages (fastKQR).
Future research directions include exploring massive parallelization strategies to further scale the algorithm and investigating adaptive preconditioning techniques to broaden its applicability to an even wider range of settings.

\bibliographystyle{abbrv}

    \appendix
	\section{Some concepts from convex analysis}\label{sec:prelim}
    This appendix collects several standard notions from convex analysis and variational analysis used in the semismooth Newton method. A convex function $f:\R^n \to (-\infty,+\infty]$ is called proper if $f(x)<+\infty$ for at least one $x$ and $f(x)>-\infty$ for every $x$. It is called closed if all its sublevel sets $\{x\mid f(x)\le \alpha\}$ are closed.

	Let $f:\R^n\to(-\infty,+\infty]$ be proper, closed, and convex, and let $\sigma>0$. The Moreau-Yosida regularization (or Moreau envelope) of $f$ is
	\begin{equation*}
		\mathcal{M}_{f}^{\sigma}(x):=\min_{z\in \mathbb{R}^n}\left\{f(z)+\frac{\sigma}{2}\|z-x\|_2^2\right\},\quad x\in \mathbb{R}^n.
	\end{equation*}
	The unique minimizer is denoted by
	\begin{equation*}
		{\rm Prox}_{f}^{\sigma}(x):=\operatornamewithlimits{argmin}_{z\in \mathbb{R}^n}\left\{f(z)+\frac{\sigma}{2}\|z-x\|_2^2\right\},
	\end{equation*}
	and the mapping ${\rm Prox}_{f}^{\sigma}$ is called the proximal mapping of $f$. A basic fact we use is that $\mathcal{M}_{f}^{\sigma}$ is continuously differentiable, with gradient
	\begin{equation}\label{def-grad-MY}
		\nabla \mathcal{M}_{f}^{\sigma}(x)=\sigma\left(x-{\rm Prox}_{f}^{\sigma}(x)\right),\quad x\in\mathbb{R}^n.
	\end{equation}
	The Moreau identity is
	\begin{equation*}
		x = {\rm Prox}_f^{\sigma^{-1}}(x) + \sigma {\rm Prox}_{f^*}^{\sigma}(\sigma^{-1}x).
	\end{equation*}
	See \cite[Chapter 1.G]{rockafellar2009variational} for further properties of the Moreau envelope and proximal mappings.

	Next we recall the notion of semismoothness. Let $F:\mathbb{R}^n \to \mathbb{R}^m$ be locally Lipschitz. By Rademacher's theorem, $F$ is differentiable almost everywhere. The Clarke generalized Jacobian of $F$ at $x\in\mathbb{R}^n$ is defined by
	\begin{equation*}
		\partial F(x):=\operatorname{conv}\left\{\lim_{k\to\infty}JF(x^k)\,:\,x^k\to x,\ \text{$F$ is differentiable at each $x^k$}\right\},
	\end{equation*}
	where $JF(x)$ denotes the Jacobian matrix of $F$ and $\operatorname{conv}(\cdot)$ denotes the convex hull. We say that $F$ is semismooth at $x$ if $F$ is directionally differentiable at $x$ and, for any $V_h\in \partial F(x+h)$,
	\[
		F(x+h) - F(x) - V_h h = o(\|h\|_2) \quad \mbox{as $h\to 0$}.
	\]
	Detailed properties of semismooth functions can be found in \cite{facchinei2007finite}.

    \section{Convergence properties of Algorithms~\ref{alg:alm} and~\ref{alg:ssnal}}\label{sec:prelim-con}
    We first state the convergence properties of the outer ALM scheme, which follow directly from \cite{rockafellar1976monotone,rockafellar1976augmented}. Since solving the ALM subproblem \eqref{def:phi_k} exactly is difficult, we introduce the following inexactness criteria, stated directly in terms of the function $\phi_k$ defined in \eqref{def:phi_k}.
	\begin{enumerate}
		\item[$(A)$] $\phi_k(\alpha^{k+1})-\inf \phi_k\leq \epsilon_k^2/(2\sigma_k),\quad \epsilon_k\ge 0,\ \sum_{k=0}^{\infty}\epsilon_k<\infty$.
		\item[$(B)$] $\phi_k(\alpha^{k+1})-\inf \phi_k\leq (\delta_k^2/(2\sigma_k))\|(\beta^{k+1},z^{k+1})-(\beta^k,z^k)\|_2^2,\quad \delta_k\ge 0,\ \sum_{k=0}^{\infty}\delta_k<\infty$.
	\end{enumerate}

	\begin{theorem}
        Let $\{(\alpha^k,v^k,\beta^k,z^k)\}$ be the sequence generated by Algorithm~\ref{alg:alm} with stopping criterion $(A)$. Then $\{(\alpha^k,v^k)\}$ is bounded and converges to the unique optimal solution of \eqref{dual-1}. In addition, $\{(\beta^k,z^k)\}$ is bounded and converges to an optimal solution associated with \eqref{primal-1}. Furthermore, if Algorithm~\ref{alg:alm} is executed with stopping criteria $(A)$ and $(B)$, then, for all $k$ sufficiently large,
		\begin{equation*}
            {\rm dist}((\beta^{k+1},z^{k+1}),\Omega)\le \theta_k\,{\rm dist}((\beta^k,z^k),\Omega),
		\end{equation*}
        where $\Omega$ denotes the solution set of $(\beta,z)$ associated with \eqref{primal-1}, $\theta_k<1$, and $\theta_k\to\theta_\infty<1$.
	\end{theorem}

    We next summarize the convergence properties of the semismooth Newton method used to solve the ALM subproblem.
    \begin{theorem}\label{thm:ssn_convergence}
    Let $\{\alpha^t\}$ be the sequence generated by Algorithm~\ref{alg:ssnal}. Then, $\{\alpha^t\}$ converges globally to the unique optimal solution $\bar{\alpha}$ of  \eqref{eq:optimality_condition}. Furthermore,   the local superlinear convergence rate holds:
    \begin{equation*}
        \|\alpha^{t+1}-\bar{\alpha}\|_2=\mathcal{O}(\|\alpha^{t}-\bar{\alpha}\|_2^{1+\iota}),
    \end{equation*}
    where the exponent $\iota\in(0,1]$ is the parameter defined in Algorithm~\ref{alg:ssnal}.
    \end{theorem}

    \section{Proof of Theorem~\ref{thm:pcg rate}}\label{app:thm3proof}
    To prove Theorem~\ref{thm:pcg rate}, we first recall the following RPCholesky error bound from \cite[Theorem~5.1]{chen2023randomly}.

	\begin{theorem}
		\label{thm:RPC error}
		Let $A$ be a positive semidefinite matrix. Fix $\ell \in\mathbb{N}$ and $\gamma>0$. The rank-$r$ column Nystr\"{o}m approximation $\widehat{A}$ produced by $r\leq n$ steps of RPCholesky, i.e., $F=\mathrm{RPCholesky}(A,r),\widehat{A} = FF^{\top}$, attains the bound
		\begin{equation*}
			\mathbb{E}[ {\rm tr}(A-\widehat{A}) ] \leq (1+\gamma){\rm tr}(A-\lfloor A\rfloor_\ell),
		\end{equation*}
		if the number $r$ of columns satisfies
		\begin{equation*}
			r\geq \frac{\ell}{\gamma}+\min\left\{\ell\log\left(\frac{1}{\gamma\eta}\right),\ell+\ell\log_+\left(\frac{2^\ell}{\gamma}\right)\right\}.
		\end{equation*}
		Here the relative error $\eta$ is defined by $\eta:={\rm tr}(A-\lfloor A\rfloor_\ell)/{\rm tr}(A)$, and $\log_+(x) = \max(\log(x),0)$.
	\end{theorem}

\begin{proof}[\textbf{Proof of Theorem~\ref{thm:pcg rate}}]
		We have $A \succeq P$, since $\widehat{K}$ is a rank-$r$ column Nystr\"{o}m approximation of $K$ thus satisfying  $ K \succeq \widehat{K}$(see \cite[Lemma~2.1]{frangella2023randomized}). Therefore, $ P^{-1/2}AP^{-1/2}\succeq I $ and $\lambda_{\min} (P^{-1/2}AP^{-1/2}) \geq 1$.
Next, by noting that
$$P^{-1/2}AP^{-1/2} = P^{-1/2}(P + K - \widehat{K})P^{-1/2} = I + P^{-1/2}(K-\widehat{K})P^{-1/2},$$
we have that
\begin{equation*}\lambda_{\max}(P^{-1/2}AP^{-1/2}) \leq
1 + \lambda_{\max}(P^{-1})  \lambda_{\max}(K-\widehat{K}) \leq
1 + {\rm tr}(K-\widehat{K}) /\mu,
\end{equation*}
		where the last inequality is due to $\lambda_{\max}(A) \leq {\rm tr}(A),\ A\in\mathbb{S}^n_{+}$ and $\lambda_{\min}(P)\geq\mu$.
		Combining the bounds of minimum and maximum eigenvalues, we can obtain the condition number bound
		\begin{equation}\label{eq:condition bound}
			\kappa\left(P^{-1/2} A P^{-1/2}\right)\leq 1+{\rm tr}(K-\widehat{K})/\mu.
		\end{equation}
		Let $\ell = {\rm rank}_{\mu}(K)$. We analyze two cases based on the sampling size $r$.

		Case 1: Trivial Bound. If
		\begin{equation*}
			n \leq \ell \left(1 +\log\left(\frac{{\rm tr} (K)}{{\rm tr}(K-\lfloor K\rfloor_\ell)}\right)\right),
		\end{equation*}
		we can set the sampling size $r=n$. This choice implies $\widehat{K}=K$, as all principal submatrices are selected. Consequently, ${\rm tr}(K-\widehat{K}) = 0$, and the bound in \eqref{eq:condition bound} simplifies to $\kappa(\cdot) \leq 1$. This case is trivial.

		Case 2: General Bound. We now consider the more general case. Let the number of samples $r$ be chosen to satisfy
		$$r \geq \ell \left(1 +\log\left(\frac{{\rm tr} (K)}{{\rm tr}(K-\lfloor K\rfloor_\ell)}\right)\right).$$
		Under this condition, Theorem~\ref{thm:RPC error} provides a bound on the expected approximation error:
		\begin{equation*}
			\mathbb{E}[{\rm tr}(K-\widehat{K})]\leq
			2{\rm tr}(K-\lfloor K\rfloor_\ell)
			=2\sum_{i=\ell+1}^n\lambda_i(K) \leq 2\mu.
		\end{equation*}
		Therefore, together with \eqref{eq:condition bound}, we have that
		\begin{equation*}
			\mathbb{E}\left[\kappa\left(P^{-1/2}AP^{-1/2}\right) \right]\leq 3.
		\end{equation*}
		By Markov's inequality, it holds that
		\begin{equation*}
			\mathbb{P}\left(\kappa\left(P^{-1/2}AP^{-1/2}\right)\leq 3/\delta\right)\geq 1-\delta.
		\end{equation*}
		Finally, it follows from \cite[Theorem 6.29]{saad2003iterative} that
		\begin{equation*}
			\frac{\|d^{\,t}-d^*\|_{A}}{\|d^*\|_{A}}
\leq 2 \left(\frac{\sqrt{3/\delta}-1}{\sqrt{3/\delta}+1}\right)^t
\leq 2 e^{-2t\sqrt{\delta/3}}
\leq 2 e^{-t\sqrt{\delta}}
\leq \varepsilon, \mbox{ for } t\geq
\delta^{-1/2}\log({2}/{\varepsilon}),
		\end{equation*}
		with the failure probability at most $\delta$.

	\end{proof}

    \end{document}